\documentclass[a4paper,10.5 pt]{article}

\usepackage[english]{babel}
\usepackage[latin1]{inputenc}
\usepackage[T1]{fontenc}
\usepackage[dvips]{graphicx}
\usepackage{amsmath}
\usepackage{amsthm}
\usepackage{amssymb}
\usepackage{mathrsfs}

\usepackage{geometry}
\usepackage{dsfont}       
\usepackage{stmaryrd}     
\usepackage{subeqnarray}
\usepackage{subfigure}
\usepackage{lipsum}
\usepackage{graphicx,color,epsfig,epstopdf}

\usepackage{array,multirow,makecell}
\setcellgapes{1pt}
\makegapedcells
\newcolumntype{R}[1]{>{\raggedleft\arraybackslash }b{#1}}
\newcolumntype{L}[1]{>{\raggedright\arraybackslash }b{#1}}
\newcolumntype{C}[1]{>{\centering\arraybackslash }b{#1}}

\newtheorem{theorem}{Theorem}

\newtheorem{corollary}{Corollary}[section]

\newtheorem{lemma}{Lemma}[section]

\newtheorem{proposition}{Proposition}[section]
\newtheorem{remark}{Remark}[section]

\numberwithin{equation}{section}

 \title{\textbf{Special values of generalized multiple Hurwitz zeta function at non-positive integers}}
  \author{
    Boualem SADAOUI \footnote{Universit\'e de Khemis Miliana, Laboratoire LESI, 44225, Khemis Miliana, Alg\'erie. \qquad \qquad \qquad \qquad \qquad \qquad \textbf{E-mail:} sadaouiboualem@gmail.com}  
  }

\begin{document}

\maketitle

\begin{abstract}
In this paper, we provide an alternative method to calculate the values of generalized multiple Hurwitz zeta function at non-positive integers by  means of \emph{Raabe}'s formula and the \textit{Bernoulli} numbers.
\end{abstract}


%%%%%%%%%%%%%%%%%%%%%%%%%%%%%%%%%%%%%%%%%%%%%%%%%%%%%%%%%%%%%%%%%%%%%%%%%%%%

\section*{Introduction and notations}
\addcontentsline{toc}{section}{Introduction and notations}
The multiple Hurwitz zeta function is defined by
\begin{equation}
\label{bb}
\zeta_n( \underline{\alpha}; s_{1},\dots, s_{n})=
\sum_{\underline{m}=(m_1,\dots,m_n) \in \mathbb{N}^{n}}{{\frac{1}{(m_{1}+\alpha)^{s_{1}}\dots(m_1+\dots+m_n+\alpha)^{s_n}}}}
\end{equation}
where $\alpha\neq 0,-1,-2,...,$ and  $(s_1,\dots,s_2)s\in \mathbb{C}^n$, 
which introduced by Akiyama and Ishikawa and proved by Akiyama and Ishikawa~\cite{akiyama02}. Matsumoto and Tanigawa proved the analytic continuation of wide class of multiple Dirichlet series and multiple Hurwitz zeta functions in~\cite{matsumoto03} and~\cite{matsumoto2003}, and the analytic continuation of the series~(\ref{bb}) is a special case of~\cite[Theorem 1]{matsumoto03}.

Our main result in this work is the values at non positive integers of the following series
\begin{equation}
\zeta_n(\underline{\alpha}; \underline{s})= \sum_{\underline{m}=(m_1,\dots,m_n) \in \mathbb{N}^{n}}{\prod_{i=1}^{n}{\frac{1}{(m_{1}+\dots +m_{i}+\alpha_i)^{s_{i}}}}}
\end{equation}
where, $\underline{\alpha}=(\alpha_1,\dots,\alpha_n) \in \mathbb{R}^n$ verified some conditions, this series is called the generalized multiple Hurwitz zeta function.

The key of this study is the use of the \emph{Raabe} formula~\cite{friedman04} which expresses the integral in terms of the sum. 

In what follows, for any elements $\underline{\textbf{x}}=(x_1,\dots,x_n)$ and $\underline{\textbf{y}}=(y_1,\dots,y_n)$ of $\mathbb{C}^{n}$ and  $\underline{\textbf{s}}= (s_{1},..., s_{n})$ denote a vector in $\mathbb{C}^{n}$  .

%%%%%%%%%%%%%%%%%%%%%%%%%%%%%%%%%%%%%%%%%%%%%%%%%%%%%%%%%%%%%%%%%%%%%%%%

\section{Main results}

For real numbers  $\underline{\alpha}=(\alpha_1,\dots,\alpha_n) \in \mathbb{R}^n$, such that, for all $1\leq i\leq n$: 
$$\alpha_i\neq 0,-1,-2,...,$$
% we put $\alpha(i)=\sum_{j=1}^{i}\alpha_j$.\\
So, for a complex $n-$tuples $\underline{s}=(s_{1}, \dots, s_{n}) \in \mathbb{C}^{n}$, we define the generalized multiple Hurwitz zeta function by
\begin{eqnarray}
	\label{Zn}
	%\begin{array}{ccc}
	\zeta_n(\underline{\alpha}; \underline{s})&:=& \zeta( \alpha_1,\dots,\alpha_n; s_{1}, \dots, s_{n})\\
	&=& \sum_{\underline{m}=(m_1,\dots,m_n) \in \mathbb{N}^{n}}{\prod_{i=1}^{n}{\frac{1}{(m_{1}+\dots +m_{i}+\alpha_i)^{s_{i}}}}}\\
	&=&\sum_{m_1>\dots>m_n\geq 0}{\prod_{i=1}^{n}{\frac{1}{(m_{i}+\alpha_i)^{s_{i}}}}}
	%\end{array}
\end{eqnarray}
and the corresponding integral function associated to the generalized multiple Hurwitz zeta function by
\begin{equation}
	\label{Yn}
	Y_n(\underline{\alpha}; \underline{s})= \int_{[0, +\infty[^{n}}{\prod_{i=1}^{n}{\frac{1}{(x_{1}+\dots +x_{i}+\alpha_i)^{s_{i}}}} d\underline{x}}.
\end{equation}
\begin{remark}
	We remark that:
	\begin{itemize}
		\item If $\underline{\alpha}=(\alpha,\alpha,\dots,\alpha)$, then the series~(\ref{Zn}) corresponding to the classical multiple Hurwitz zeta function.
		\item If $\underline{\alpha}=(1,1,\dots,1)$, then the series~(\ref{Zn}) corresponding to the multiple zeta function.
	\end{itemize}
\end{remark}

For the meromorphic continuation of the integral~(\ref{Yn}) and the series~(\ref{Zn}), we refer the reader to the work~\cite{matsumoto03}.

We first give well-known elementary result for the integral function.
\begin{lemma}\

\label{theo1}
 Let $\underline{\textbf{N}}= (N_{1},\dots,N_{n})$ be a point of $\mathbb{N}^{n}$,
\newline
\begin{enumerate}
\item[(1)] The point $(\underline{\textbf{s}}=-\underline{\textbf{N}})$ is a polar divisor for the function $Y_n(\underline{\alpha}; \underline{\textbf{s}})$ if and only if there exists a $\underline{k}=(k_2,..., k_n) \in \mathbb{N}^{n-1}$ such that
\begin{equation}
(s_{n}-1)(s_{n}+s_{n-1}-2+k_{n})...\left(\sum_{i=1}^{n}{s_{i}}-n+\sum_{i=2}^{n}{k_{i}}\right)= \prod_{j=1}^{n}{\left(\sum_{i=j}^{n}{s_{i}}-n+j-1+\sum_{i=j+1}^{n}{k_{i}}\right)}=0.
\end{equation}
\item[(2)] If $(\underline{\textbf{s}}=-\underline{\textbf{N}})$ is not a polar divisor for the integral function, then the value of this function at this point exists and is given by
\begin{eqnarray*}
	&Y_n(\underline{\alpha}; \underline{\textbf{s}})= (-1)^{n}&\\ &\sum_{\underline{k}=(k_{2},...,k_{n}) \in \mathbb{N}^{n-1}}{\frac{\left(N_{n}+1 \atop k_{n}\right) \left(N_{n}+N_{n-1}+2-k_{n} \atop k_{n-1}\right)...\left(\sum_{i=2}^{n}{N_{i}}+n-\sum_{i=3}^{n}{k_{i}}\atop k_{2}\right) \alpha_1^{\left(-\sum_{i=1}^{n}{s_{i}}+n-\sum_{i=2}^{n}{k_{i}}\right)}}{\prod_{j=1}^{n} \left(\sum_{i=j}^{n}{N_{i}}+n-j+1-\sum_{i=j+1}^{n}{k_{i}}\right)}\;\prod_{j=2}^{n} \alpha_j^{k_j}}
\end{eqnarray*}
with
\begin{equation*}
T(\underline{\textbf{N}}):= \left\{\underline{k}=(k_{2},...,k_{n}) \in \mathbb{N}^{n-1}: \quad 0\leq k_{j}\leq \sum_{i=j}^{n}{N_{i}}+n-j+1-\sum_{i=j+1}^{n}{k_{i}}, \; \forall\; 2\leq j\leq n\right\}.
\end{equation*}
\end{enumerate}
\end{lemma}

We give now a similar result for the generalized multiple Hurwitz zeta function.
\begin{theorem}\

\label{theo2}
 Let $\underline{\textbf{N}}= (N_{1},\dots,N_{n})$ a point of $\mathbb{N}^{n}$, if  the point $(\underline{\textbf{s}}=-\underline{\textbf{N}})$ is not a polar divisor for the integral function $Y_n(\underline{\alpha}; \underline{\textbf{s}})$, then the value of the generalized multiple Hurwitz zeta function $\zeta_n(\underline{\alpha}; \underline{\textbf{s}})$ at the point $(\underline{\textbf{s}}=-\underline{\textbf{N}})$ exists and is given by
\begin{equation}
\begin{array}{ccc}
&\zeta_n(\underline{\alpha}; -\underline{\textbf{N}})=(-1)^n&\\
& \sum_{\underline{k}=(k_2,...,k_n) \in \mathbb{N}^{n-1}}{\sum_{\underline{v}=(v_1,...,v_n)\in \mathbb{N}^{n}\atop v_j\leq k_j\; \forall\; 2\leq j\leq n; v_1\leq \left(\sum_{i=1}^{n}{N_{i}}+n-\sum_{i=2}^{n}{k_{i}}\right)}A(-\underline{\textbf{N}})\; B_{\underline{v}} \prod_{j=1}^{n}{\frac{1 }{\left(\sum_{i=j}^{n}{N_{i}}+n-j+1-\sum_{i=j+1}^{n}{k_{i}}\right)}}}
\end{array}
\end{equation}
with
\begin{equation}
A(-\underline{\textbf{N}})= \left(\sum_{i=1}^{n}{N_{i}}+n-\sum_{i=2}^{n}{k_{i}}\atop v_1\right) \alpha^{\left(\sum_{i=1}^{n}{N_{i}}+n-\sum_{i=2}^{n}{v_{i}}\right)}\prod_{j=2}^{n}{\left(\sum_{i=j}^{n}{N_{i}}+n-j+1-\sum_{i=j+1}^{n}{k_{i}} \atop k_{j} \right)\left(k_j \atop v_j\right) }.
\end{equation}
and
\begin{equation*}
T(\underline{\textbf{N}}):= \left\{\underline{k}=(k_2,...,k_n) \in \mathbb{N}^{n-1}: \quad 0\leq k_{j}\leq \sum_{i=j}^{n}{N_{i}}+n-j+1-\sum_{i=j+1}^{n}{k_{i}}, \; \forall\; 2\leq j\leq n\right\}.
\end{equation*}
and
\begin{equation*}
B_{\underline{v}}= \prod_{j=1}^{n}{B_{v_{j}}}
\end{equation*}
where $B_{v_{j}}$ is the ${v_{j}}^{-th}$ Bernoulli number.
\end{theorem}
%\subsubsection*{Proof of Theorem~\ref{theo2} in the case $n=1$:}

%In this part, we give the proof of our main result for $n=1$ as a warm up for the proof for larger $n$.
%We have
%\begin{equation}
%\zeta(s)= \sum_{n>0}{\frac{1}{n^{s}}}= Z(s).
%\end{equation}
%For $a \in \mathbb{R}_{+}$ we set
%\begin{equation}
%Y_{a}(s)= \int_{1}^{+\infty}{(x+a)^{-s} dx}
%\end{equation}
%which for $\Re(s)>1$ reads
%\begin{equation}
%Y_{a}(s)= \int_{1}^{+\infty}{(x+a)^{-s} dx}= \frac{(1+a)^{-s+1}}{s-1}
%\end{equation}
%Thus, for all $N \in \mathbb{N}$:
%\begin{equation}
%Y_{a}(-N)= -\frac{(1+a)^{N+1}}{N+1}= \frac{-1}{N+1}\sum_{k=0}^{N+1}{\left(N+1 \atop k\right) a^{k}}.
%\end{equation}
%Then, Proposition~\ref{pro4} of Section~\ref{4} shows that
%\begin{equation}
%\zeta(-N)= Z(-N)= \frac{-1}{N+1}\sum_{k=0}^{N+1}{\left(N+1 \atop k\right) B_{k}}
%\end{equation}
%where,  $B_{k}$ is the $k^{-th}$ Bernoulli number, which ends the proof of Theorem~\ref{theo2} for $n=1$.
%\\
%Now, we recall the elementary result
%\begin{equation}
%(N+1) B_{N}= -\sum_{k=0}^{N-1}{\left(N+1 \atop k\right) B_{k}}.
%\end{equation}
%Finally, we obtain the known result
%\begin{equation}
%\zeta(-N)= Z(-N)= -\frac{B_{N+1}}{N+1}.
%\end{equation}

%%%%%%%%%%%%%%%%%%%%%%%%%%%%%%%%%%%%%%%%%%%%%%%%%%%%%%%%%%%%%%%%%%%%%%%%%

\section{Proof of lemma~\ref{theo1}}

	Let the integral function
	\begin{equation}
	\label{equa11}
	Y_n(\underline{\alpha}; \underline{\textbf{s}})=\int_{[0, +\infty[^{n}} {\prod_{i=1}^{n}{(x_{1}+ \dots+ x_{i}+\alpha_i)^{- s_{i}}}\; d{\underline{x}}}.
	\end{equation}
	If we use the following change of variables:
	\begin{equation}
	y_i=x_1+\dots+x_i+\alpha_i
	\end{equation}
	for all $1\leq i\leq n$, we find
	\begin{equation}
	%\label{equa11}
		Y_n(\underline{\alpha}; \underline{\textbf{s}})=\int_{\prod_{i=1}^{n}[\alpha_i, +\infty[^{n}} {\prod_{i=1}^{n}{(y_{1}+ \dots+ y_{i})^{- s_{i}}}\; d{\underline{x}}}.
	\end{equation}
	Now, using the following change of variables:
	\begin{equation}
	z_{i}= y_{1}+\dots+y_{i}-\sum_{j=2}^{i}\alpha_j
	\end{equation}
	for all $1\leq i\leq n$.
	This change gives
	\begin{equation}
	\left\{\begin{array}{ccc}	
		y_{1}=z_{1}&\\
		y_{i}= z_{i}-z_{i-1}+\alpha_i,& \forall\; 2\leq i\leq n
	\end{array}\right.
	\end{equation}
	
	Since $\underline{y}= (y_{1},..., y_{n}) \in \prod_{i=1}^{n}[\alpha_i, +\infty[$, this gives
	\begin{equation}
	\underline{z} \in V_{n}= \left\{\underline{z} \in \mathbb{R}^{n}:\quad \alpha_1\leq z_{1}\leq z_{2}\leq \dots\leq z_{n} \right\}
	\end{equation}
	and, we find
	\begin{equation}
	Y_n(\underline{\alpha};  \underline{\textbf{s}})=\int_{V_{n}} {\prod_{i=1}^{n}{(z_{i}+\sum_{j=2}^{i}\alpha_j)^{- s_{i}}}\; d{\underline{z}}}.
	\end{equation}
	This integral can be rewritten as follows.
	\begin{equation}
	Y_n(\underline{\alpha};  \underline{\textbf{s}})=\int_{V_{n-1}} {\prod_{i=1}^{n-1}{(z_{i}+\sum_{j=2}^{i}\alpha_j)^{- s_{i}}}\;\left(\int_{z_{n-1}}^{+\infty}(z_{n}+\sum_{j=2}^{n}\alpha_j)^{- s_{n}}\;dz_{n}\right) dz_{1}...dz_{n-1}}
	\end{equation}
	with
	\begin{equation}
	\begin{array}{ccc}
	\int_{z_{n-1}}^{+\infty}{(z_{n}+\sum_{j=2}^{n}\alpha_j)^{- s_{n}}\; dz_{n}}&=& \frac{(z_{n-1}+\sum_{j=2}^{n-1}\alpha_j)^{-s_{n}+1}}{s_{n}-1}\left(1+\frac{\alpha_n}{z_{n-1}+\sum_{j=2}^{n-1}\alpha_j}\right)^{-s_{n}+1}\\
	&=&\sum_{k_{n}\in \mathbb{N}}{\left(-s_{n}+1 \atop k_{n}\right) \frac{(z_{n-1}+\sum_{j=2}^{n-1}\alpha_j)^{-s_{n}+1-k_{n}}}{s_{n}-1}\; \alpha_n^{k_n}}
	\end{array}
	\end{equation}
	if and only if $\Re(s_{n})-1>0$.
	\newline
	Inductively on $n$, we find
	\begin{eqnarray*}
Y_n(\underline{\alpha}; \underline{\textbf{s}})&=& \sum_{\underline{k}=(k_{2},...,k_{n}) \in \mathbb{N}^{n-1}}{\frac{\left(-s_{n}+1 \atop k_{n}\right) \left(-s_{n}-s_{n-1}+2-k_{n} \atop k_{n-1}\right)...\left(-\sum_{i=2}^{n}{s_{i}}+n-\sum_{i=3}^{n}{k_{i}}\atop k_{2}\right) \alpha_1^{\left(-\sum_{i=1}^{n}{s_{i}}+n-\sum_{i=2}^{n}{k_{i}}\right)}}{(s_{n}-1)(s_{n}+s_{n-1}-2+k_{n})...\left(\sum_{i=1}^{n}{s_{i}}-n+\sum_{i=2}^{n}{k_{i}}\right)}\;\prod_{j=2}^{n} \alpha_j^{k_j}}\\
%&=& \sum_{\underline{k}=(k_{2},...,k_{n}) \in \mathbb{N}^{n-1}}{\frac{\left(-s_{n}+1 \atop k_{n}\right) \left(-s_{n}-s_{n-1}+2-k_{n} \atop k_{n-1}\right)...\left(-\sum_{i=2}^{n}{s_{i}}+n-\sum_{i=3}^{n}{k_{i}}\atop k_{2}\right) \alpha^{\left(-\sum_{i=1}^{n}{s_{i}}+n\right)}}{(s_{n}-1)(s_{n}+s_{n-1}-2+k_{n})...\left(\sum_{i=1}^{n}{s_{i}}-n+\sum_{i=2}^{n}{k_{i}}\right)}}
	\end{eqnarray*}
	if and only if for all $1\leq i\leq n-1$
	\begin{equation}
	\Re\left(\sum_{i=1}^{n}{s_{i}}\right)-n+j-1+\sum_{i=2}^{n}{k_{i}}>0
	\end{equation}
	and  
	\begin{equation}
	\Re(s_{n})-1>0.
	\end{equation}
	\newline
	Therefore, for any point $\underline{\textbf{N}}=(N_{1},\dots,N_{n}) \in \mathbb{N}^{n}$
	\begin{enumerate}
		\item[1)] The point $(\underline{\textbf{s}}= -\underline{\textbf{N}})$ is a polar divisor for the function $Y_n(\underline{\alpha};  \underline{\textbf{s}})$ if there exists a $\underline{k}=(k_{2},...,k_{n}) \in \mathbb{N}^{n-1}$ such that
		\begin{equation}
		(s_{n}-1)(s_{n}+s_{n-1}-2+k_{n})...\left(\sum_{i=1}^{n}{s_{i}}-n+\sum_{i=2}^{n}{k_{i}}\right)= \prod_{j=1}^{n}{\left(\sum_{i=j}^{n}{s_{i}}-n+j-1+\sum_{i=j+1}^{n}{k_{i}}\right)}=0.
		\end{equation}
		\item[2)] If $(\underline{\textbf{s}}=-\underline{\textbf{N}})$ is not a polar divisor we get
		\begin{equation}
		\left(N_{n}+1 \atop k_{n}\right)...\left(\sum_{i=2}^{n}{N_{i}}+n-\sum_{i=3}^{n}{k_{i}}\atop k_{2}\right)= \prod_{j=2}^{n}{\left(\sum_{i=j}^{n}{N_{i}}+n-j+1-\sum_{i=j+1}^{n}{k_{i}} \atop k_{j} \right)}=0 
		\end{equation}
		if and only if there exists an $\underline{k}=(k_2,...,k_n) \in \mathbb{N}^{n-1}$ and $2\leq j\leq n$, such that
		\begin{equation*}
		k_{j} > \sum_{i=j}^{n}{N_{i}}+n-j+1-\sum_{i=j+1}^{n}{k_{i}}.
		\end{equation*}
		Let 
		\begin{equation}
		T(\underline{\textbf{N}}):= \left\{\underline{k}=(k_{2},...,k_{n}) \in \mathbb{N}^{n-1}: \quad 0\leq k_{j}\leq \sum_{i=j}^{n}{N_{i}}+n-j+1-\sum_{i=j+1}^{n}{k_{i}}, \; \forall\; 2\leq j\leq n\right\}
		\end{equation}
		which is finite, then
		\begin{eqnarray*}
&Y_n(\underline{\alpha}; \underline{\textbf{s}})= (-1)^{n}&\\ &\sum_{\underline{k}=(k_{2},...,k_{n}) \in \mathbb{N}^{n-1}}{\frac{\left(N_{n}+1 \atop k_{n}\right) \left(N_{n}+N_{n-1}+2-k_{n} \atop k_{n-1}\right)...\left(\sum_{i=2}^{n}{N_{i}}+n-\sum_{i=3}^{n}{k_{i}}\atop k_{2}\right) \alpha_1^{\left(-\sum_{i=1}^{n}{s_{i}}+n-\sum_{i=2}^{n}{k_{i}}\right)}}{\prod_{j=1}^{n} \left(\sum_{i=j}^{n}{N_{i}}+n-j+1-\sum_{i=j+1}^{n}{k_{i}}\right)}\;\prod_{j=2}^{n} \alpha_j^{k_j}}
%	&=&(-1)^{n}\alpha^{\left(\sum_{i=1}^{n}{N_{i}}+n\right)}\sum_{\underline{k} \in T(\underline{\textbf{N}})}{{\frac{\prod_{j=2}^{n} \left(\sum_{i=j}^{n}{N_{i}}+n-j+1-\sum_{i=j+1}^{n}{k_{i}} \atop k_{j} \right)}{\prod_{j=1}^{n} \left(\sum_{i=j}^{n}{N_{i}}+n-j+1-\sum_{i=j+1}^{n}{k_{i}}\right)}}}
		\end{eqnarray*}
	\end{enumerate}

%%%%%%%%%%%%%%%%%%%%%%%%%%%%%%%%%%%%%%%%%%%%%%%%%%%%%%%%%%%%%%%%%

\section{An intermediate approximation}
For $\underline{a}=(a_1,...,a_n) \in \mathbb{R}_{+}^{n}$ and  $\underline{\textbf{s}}=(s_{1},...,s_{n}) \in \mathbb{C}^{n}$, we define the function
\begin{equation}
Y_{n,\underline{a}}(\underline{\alpha};  \underline{\textbf{s}})= \int_{\prod_{i=1}^{n}[\alpha_i,+\infty[^{n}}{\prod_{i=1}^{n}{(x_{1}+...+x_{i}+a_{1}+...+a_{i})^{-s_i}}\; d\underline{x}}.
\end{equation}
We prove the following useful result.
\begin{proposition}
\label{inter-estim}
Let $\underline{\textbf{N}}=(N_1,..., N_n)$ a point of $\mathbb{N}^{n}$, then we have for $\underline{a} \in \mathbb{R}^{+}$
\begin{equation}
\label{Ya}
\begin{array}{ccc}
&Y_{n,\underline{a}}(\underline{\alpha};  -\underline{\textbf{N}})=(-1)^n&\\
& \sum_{\underline{k}=(k_2,...,k_n) \in \mathbb{N}^{n-1}}{\sum_{\underline{v}=(v_1,...,v_n)\in \mathbb{N}^{n}\atop v_j\leq k_j\; \forall\; 2\leq j\leq n; v_1\leq \left(-\sum_{i=1}^{n}{s_{i}}+n-\sum_{i=2}^{n}{k_{i}}\right)}\frac{A(-\underline{\textbf{N}})\; a_{1}^{v_1}}{\left(\sum_{i=1}^{n}{N_{i}}+n-\sum_{i=2}^{n}{k_{i}}\right)}\prod_{j=2}^{n}{\frac{ a_j^{v_j}}{\left(\sum_{i=j}^{n}{N_{i}}+n-j+1-\sum_{i=j+1}^{n}{k_{i}}\right)}}}.
\end{array}
\end{equation}
with
\begin{equation}
A(-\underline{\textbf{N}})= \left(\sum_{i=1}^{n}{N_{i}}+n-\sum_{i=2}^{n}{k_{i}}\atop v_1\right) \alpha_1^{\left(\sum_{i=1}^{n}{N_{i}}+n-\sum_{i=2}^{n}{v_{i}}\right)} \prod_{j=2}^{n}{\left(\sum_{i=j}^{n}{N_{i}}+n-j+1-\sum_{i=j+1}^{n}{k_{i}} \atop k_{j} \right)\left(k_j \atop v_j\right) \alpha_j^{k_j-v_j}}.
\end{equation}
and
\begin{equation}
T(\underline{\textbf{N}}):= \left\{\underline{k}=(k_2,...,k_n) \in \mathbb{N}^{n-1}: \quad 0\leq k_{j}\leq \sum_{i=j}^{n}{N_{i}}+n-j+1-\sum_{i=j+1}^{n}{k_{i}}, \; \forall\; 2\leq j\leq n\right\}
\end{equation}
\end{proposition}

\begin{proof}
Let $\underline{a} \in \mathbb{R}^{n}_{+}$, such that for all $\underline{x}=(x_1,...,x_n) \in [\alpha, +\infty[^{n}$ and   
for all $1\leq i\leq n$
\begin{equation}
\label{cond1}
\frac{\alpha_i+a_i}{x_1+...+x_{i-1}+a_1+...+a_{i-1}}<1,
\end{equation} 
we have 
\begin{equation}
Y_{n,\underline{a}}(\underline{\alpha};  \underline{\textbf{s}})= \int_{\prod_{i=1}^{n}[\alpha_i,+\infty[^{n}}{\prod_{i=1}^{n}{(x_{1}+...+x_{i}+a_{1}+...+a_{i})^{-s_i}}\; d\underline{x}}.
\end{equation}
This integral can be written as follows
\begin{eqnarray*}
&Y_{n,\underline{a}}(\underline{\alpha};  \underline{\textbf{s}})=&\\ &\int_{\prod_{i=1}^{n-1}[\alpha_i,+\infty[^{n-1}}{\prod_{i=1}^{n-1}{(x_{1}+...+x_{i}+a_{1}+...+a_{i})^{-s_i}}\;\left(\int_{\alpha_n}^{+\infty}{(x_1+...+x_n+a_1+...+a_n)^{-s_n}\; dx_n}\right) dx_1...dx_{n-1}}\\
\end{eqnarray*}
Since for $\Re(s_n)>1$ we have
\begin{equation}
\int_{\alpha_n}^{+\infty}{(x_1+...+x_n+a_1+...+a_n)^{-s_n}\; dx_n}= \frac{(x_1+...+x_{n-1}+a_1+...a_{n-1}+\alpha_n+a_n)^{-s_n+1}}{s_n-1}
\end{equation}
condition~(\ref{cond1}) yields
\begin{equation}
\int_{\alpha_n}^{+\infty}{(x_1+...+x_n+a_1+...+a_n)^{-s_n}\; dx_n}= \sum_{k_n \in \mathbb{N}}\left(-s_n+1 \atop k_n\right)\frac{(\alpha_n+a_n)^{k_n}}{s_n-1} (x_1+...+x_{n-1}+a_1+...a_{n-1})^{-s_n+1-k_n}.
\end{equation}
If for $1\leq j\leq n-1$
\begin{equation}
\left(\sum_{i=j}^{n}{\Re(s_{i})}-n+j-1+\sum_{i=j+1}^{n}{k_{i}}\right)>0
\end{equation}
then inductively we find
\begin{equation}
\begin{array}{ccc}
&Y_{n,\underline{a}}(\underline{\alpha};  \underline{\textbf{s}})= (-1)^n&\\
& \sum_{\underline{k}=(k_2,...,k_n) \in \mathbb{N}^{n-1}}{\frac{(\alpha_1+a_1)^{-\sum_{i=1}^{n}{s_{i}}+n-\sum_{i=2}^{n}{k_{i}}}}{\left(-\sum_{i=1}^{n}{s_{i}}+n-\sum_{i=2}^{n}{k_{i}}\right)}\prod_{j=2}^{n}{\left(-\sum_{i=j}^{n}{s_{i}}+n-j+1-\sum_{i=j+1}^{n}{k_{i}} \atop k_{j} \right)\frac{(\alpha_j+a_j)^{k_j} }{ \left(-\sum_{i=j}^{n}{s_{i}}+n-j+1-\sum_{i=j+1}^{n}{k_{i}}\right)}}}.&
\end{array}
\end{equation}
But, for all $2\leq j\leq n$ we have
\begin{equation}
(\alpha_j+a_j)^{k_j}= \sum_{v_j \in \mathbb{N}\atop v_j\leq k_j}{\left(k_j\atop v_j\right)\alpha_j^{k_j-v_j} a_{j}^{v_j}}
\end{equation}
and
\begin{equation}
(\alpha_1+a_1)^{-\sum_{i=1}^{n}{s_{i}}+n-\sum_{i=2}^{n}{k_{i}}}= \sum_{v_1 \in \mathbb{N}\atop v_1\leq \left(-\sum_{i=1}^{n}{s_{i}}+n-\sum_{i=2}^{n}{k_{i}}\right)}{\left(-\sum_{i=1}^{n}{s_{i}}+n-\sum_{i=2}^{n}{k_{i}}\atop v_1\right)\alpha_1^{\left(-\sum_{i=1}^{n}{s_{i}}+n-\sum_{i=2}^{n}{k_{i}}\right)} a_{1}^{v_1}}
\end{equation}
which yields
\begin{equation}
\begin{array}{ccc}
&Y_{n,\underline{a}}(\underline{\alpha};  \underline{\textbf{s}})=(-1)^n&\\
& \sum_{\underline{k}=(k_2,...,k_n) \in \mathbb{N}^{n-1}}{\sum_{\underline{v}=(v_1,...,v_n)\in \mathbb{N}^{n}\atop v_j\leq k_j\; \forall\; 2\leq j\leq n; v_1\leq \left(-\sum_{i=1}^{n}{s_{i}}+n-\sum_{i=2}^{n}{k_{i}}\right)}\frac{A(\underline{\textbf{s}})\; a_{1}^{v_1}}{\left(-\sum_{i=1}^{n}{s_{i}}+n-\sum_{i=2}^{n}{k_{i}}\right)}\prod_{j=2}^{n}{\frac{ a_j^{v_j}}{\left(-\sum_{i=j}^{n}{s_{i}}+n-j+1-\sum_{i=j+1}^{n}{k_{i}}\right)}}}.
\end{array}
\end{equation}
with
\begin{equation}
A(\underline{\textbf{s}})= \left(-\sum_{i=1}^{n}{s_{i}}+n-\sum_{i=2}^{n}{k_{i}}\atop v_1\right) \alpha_1^{\left(-\sum_{i=1}^{n}{s_{i}}+n-\sum_{i=2}^{n}{v_{i}}\right)} \prod_{j=2}^{n}{\left(-\sum_{i=j}^{n}{s_{i}}+n-j+1-\sum_{i=j+1}^{n}{k_{i}} \atop k_{j} \right)\left(k_j \atop v_j\right) \alpha_j^{k_j-v_j}}.
\end{equation}
Setting $\underline{\textbf{s}}=- \underline{\textbf{N}}=-(N_1,...,N_n) \in \mathbb{N}^{n}$
yields~(\ref{Ya}) and ends the proof of Proposition~\ref{inter-estim}. 
\end{proof} 

%%%%%%%%%%%%%%%%%%%%%%%%%%%%%%%%%%%%%%%%%%%%%%%%%%%%%%%%%%%%%%%%%%%%%%%%

\section{Proof of Theorem~\ref{theo2}}
\label{4}
The proof relies on the \textit{Raabe} formula~\cite{friedman04}, which expresses the integral in terms of the sum.
\begin{proposition}\

\label{proRaab}
 
\begin{enumerate}
\item[(1)] \textit{Raabe} formula:\

for all $\underline{\textbf{s}} \in \mathbb{C}^{n}$, outside the possible polar divisors of $Y_{n}(\underline{\alpha}; \underline{\textbf{s}})$, we have:
\begin{equation}
 Y_{n}(\underline{\alpha}; \underline{\textbf{s}})= \int_{\underline{\textbf{t}} \in [0,1]^{n}}{\zeta_{n,\underline{\textbf{t}}}(\underline{\alpha}; \underline{\textbf{s}})\; d\underline{\textbf{t}}}
\end{equation}
where:  $$\zeta_{n,\underline{\textbf{t}}}(\underline{\alpha}; \underline{\textbf{s}})= \sum_{\underline{m} \in \mathbb{N}^{n}}{\prod_{i=1}^{n}{\frac{1}{\left((m_{1}+t_{1})+\dots +(m_{i}+t_{i}+\alpha_i)\right)^{s_{i}}}}}$$ and $d{\underline{\textbf{t}}}$ is the Lebesgue measure on $\mathbb{R}^{n}$.
\item[(2)] For a fixed point $\underline{\textbf{N}}= (N_{1},..., N_{n})$ in $\mathbb{N}^{n}$ the maps $\displaystyle{\underline{\textbf{a}} \mapsto Y_{n,\underline{a}}(\underline{\alpha}; -\underline{\textbf{N}})}$ and 
$\displaystyle{\underline{\textbf{a}} \mapsto \zeta_{n,\underline{\textbf{t}}}(\underline{\alpha}; -\underline{\textbf{N}})}$ 
are polynomials in $\underline{\textbf{a}}= (a_{1},..., a_{n}) \in \mathbb{R}_{+}^{n}$. 
\end{enumerate}
\end{proposition}
\begin{proof}\
\begin{enumerate}
\item[(1)]
Let $\underline{\textbf{s}} \in \mathbb{C}^{n}$ be chosen in such a way that the integral function and the generalized multiple Hurwitz zeta function are  absolutely convergent.
\newline
Thus, for $\underline{\textbf{t}} \in \mathbb{R}_{+}^{n}$, we have:
\begin{eqnarray*}
%\begin{array}{ccc}
\int_{[0,1]^{n}}{\zeta_{n,\underline{\textbf{t}}}(\underline{\alpha}; \underline{\textbf{s}})\; d\underline{\textbf{t}}} &=& \int_{[0,1]^{n}}{\sum_{\underline{\textbf{m}} \in \mathbb{N}^{n}}{\prod_{i=1}^{n}{(t_{1}+\dots +t_{i}+ m_{1}+\dots +m_{i}+\alpha_i)^{-s_{i}}\; d\underline{\textbf{t}}}}}\\
&=& \sum_{\underline{\textbf{m}}=(m_{1},..., m_{n}) \in \mathbb{N}^{n}}{\quad\int_{\prod_{i=1}^{n}{[m_{i},m_{i}+1]}}{\quad\prod_{i=1}^{n}{(t_{1}+\dots +t_{i}+ m_{1}+\dots +m_{i}+\alpha_i)^{-s_{i}}}}}\; d\underline{\textbf{t}}\\
&=& \int_{[0, +\infty[^{n}}{\prod_{i=1}^{n}{(x_{1}+\dots +x_{i}+\alpha_i)^{-s_{i}}\; d\underline{\textbf{x}}}}= Y_{n}(\underline{\alpha}; \underline{\textbf{s}}).
%\end{array}
\end{eqnarray*}
\newline
This last equality which is verified for all $\underline{\textbf{s}} \in \mathbb{C}^{n}$ follows by analytic continuation outside the polar divisors.
\item[(2)]
follows from~(\ref{Ya}) combined with the \textit{Raabe} formula.
\end{enumerate}
\end{proof}
 
\begin{lemma}[\cite{friedman}]\

 Let $P$ and $Q$ to be two polynomials in $n$ variables linked by
\begin{equation}
\label{eqP1}
 P(\underline{\textbf{a}})= \int_{\underline{\textbf{t}} \in [0,1]^{n}}{Q(\underline{\textbf{a}}+ \underline{\textbf{t}})\; d\underline{\textbf{t}}}.
\end{equation}
Write out 
\begin{equation} 
 \label{eqP2}
P(\underline{\textbf{a}})= P(a_{1},..., a_{n})= \sum_{\underline{\textbf{L}}}{h_{\underline{\textbf{L}}}\; \prod_{i=1}^{n}{a_{i}^{L_{i}}}}
\end{equation}
where $h_{\underline{\textbf{L}}} \in \mathbb{C}$ and  $\underline{\textbf{L}}= (L_{1},..., L_{n}) \in \mathbb{N}^{n}$ ranges over a finite set of multi-index. Then
\begin{equation}
\label{eqQ}
 Q(\underline{\textbf{a}}) = Q(a_{1},..., a_{n})= \sum_{\underline{\textbf{L}}}{h_{\underline{\textbf{L}}}\; \prod_{i=1}^{n}{B_{L_{i}}(a_{i})}}
\end{equation}
where the $B_{L_{i}}(a_{i})$ are the Bernoulli polynomials~\cite{apostol76}.
\newline
Conversely, if $Q$ is given by~(\ref{eqQ}), then the relations~(\ref{eqP1}) and (\ref{eqP2}) yield equivalent formulas for  the polynomial $P$.
\end{lemma}
\begin{proposition}\
\label{pro4}
 If we write out the polynomial  $Y_{\underline{\textbf{a}}}(\underline{\alpha}; -\underline{\textbf{N}})$ as a sum of monomials,
\begin{equation*}
 Y_{\underline{\textbf{a}}}(\underline{\alpha}; -\underline{\textbf{N}})= \sum_{\underline{\textbf{L}}}{C_{\underline{\textbf{L}}}\; 
\underline{\textbf{a}}^{\underline{\textbf{L}}}}
\end{equation*}
with $\displaystyle{\underline{\textbf{a}}^{\underline{\textbf{L}}}= \prod_{i=1}^{n}{a_{i}^{L_{i}}}}$ and $C_{\underline{\textbf{L}}}= C_{\underline{\textbf{L}}}(\underline{\textbf{N}}) \in \mathbb{C}$.
\newline
Then
\begin{equation*}
 \zeta_n(\underline{\alpha}; -\underline{\textbf{N}})= \sum_{\underline{\textbf{L}}}{C_{\underline{\textbf{L}}}\; 
B_{\underline{\textbf{L}}}}
\end{equation*}
where $\displaystyle{B_{\underline{\textbf{L}}}= \prod_{i=1}^{n}}{B_{L_{i}}}$ is a product of Bernoulli numbers.
\newline
More generally, for $\underline{\textbf{a}}= (a_{1},..., a_{n}) \in \mathbb{R}_{+}^{n}$, we have:
\begin{equation*}
  \zeta_{n,\underline{\textbf{a}}}(\underline{\alpha}; -\underline{\textbf{N}})= \sum_{\underline{\textbf{L}}}{C_{\underline{\textbf{L}}}\; 
B_{\underline{\textbf{L}}}(\underline{\textbf{a}})}
\end{equation*}
where $\displaystyle{B_{\underline{\textbf{L}}}(\underline{\textbf{a}})= \prod_{i=1}^{n}}{B_{L_{i}}}(a_{i})$ is a product of Bernoulli numbers.
\end{proposition}
\begin{proof}\
It follows from the above lemma, with $P(\underline{a})= Y_{n,\underline{a}}(\underline{\alpha}; -\underline{\textbf{N}})$
and $\displaystyle{Q(\underline{a})= \zeta_{n,\underline{a}}(\underline{\alpha}; -\underline{\textbf{N}})}$.
\end{proof}
\subsection{Proof of Theorem~\ref{theo2}:}
Relation~(\ref{Ya}) shows that for all $\underline{a} \in \mathbb{R}_{+}^{n}$
\begin{equation}
%\label{Ya}
\begin{array}{ccc}
&Y_{n,\underline{a}}(\underline{\alpha};  -\underline{\textbf{N}})=(-1)^n&\\
& \sum_{\underline{k}=(k_2,...,k_n) \in \mathbb{N}^{n-1}}{\sum_{\underline{v}=(v_1,...,v_n)\in \mathbb{N}^{n}\atop v_j\leq k_j\; \forall\; 2\leq j\leq n; v_1\leq \left(-\sum_{i=1}^{n}{s_{i}}+n-\sum_{i=2}^{n}{k_{i}}\right)}\frac{A(-\underline{\textbf{N}})\; a_{1}^{v_1}}{\left(\sum_{i=1}^{n}{N_{i}}+n-\sum_{i=2}^{n}{k_{i}}\right)}\prod_{j=2}^{n}{\frac{ a_j^{v_j}}{\left(\sum_{i=j}^{n}{N_{i}}+n-j+1-\sum_{i=j+1}^{n}{k_{i}}\right)}}}.
\end{array}
\end{equation}
with
\begin{equation}
A(-\underline{\textbf{N}})= \left(\sum_{i=1}^{n}{N_{i}}+n-\sum_{i=2}^{n}{k_{i}}\atop v_1\right) \alpha_1^{\left(\sum_{i=1}^{n}{N_{i}}+n-\sum_{i=2}^{n}{v_{i}}\right)} \prod_{j=2}^{n}{\left(\sum_{i=j}^{n}{N_{i}}+n-j+1-\sum_{i=j+1}^{n}{k_{i}} \atop k_{j} \right)\left(k_j \atop v_j\right) \alpha_j^{k_j-v_j}}.
\end{equation}
and
\begin{equation}
T(\underline{\textbf{N}}):= \left\{\underline{k}=(k_2,...,k_n) \in \mathbb{N}^{n-1}: \quad 0\leq k_{j}\leq \sum_{i=j}^{n}{N_{i}}+n-j+1-\sum_{i=j+1}^{n}{k_{i}}, \; \forall\; 2\leq j\leq n\right\}.
\end{equation}
Setting, 
\begin{equation}
\underline{a}^{\underline{v}}= \prod_{j=1}^{n}{a_j^{v_j}}
\end{equation}
this gives
\begin{equation}
%\label{Ya}
\begin{array}{ccc}
&Y_{n,\underline{a}}(\underline{\alpha}; -\underline{\textbf{N}})=(-1)^{n}&\\
& \sum_{\underline{k}=(k_2,...,k_n) \in \mathbb{N}^{n-1}}{\sum_{\underline{v}=(v_1,...,v_n)\in \mathbb{N}^{n}\atop v_j\leq k_j\; \forall\; 2\leq j\leq n; v_1\leq \left(\sum_{i=1}^{n}{N_{i}}+n-\sum_{i=2}^{n}{k_{i}}\right)}A(-\underline{\textbf{N}})\; \underline{a}^{\underline{v}} \prod_{j=1}^{n}{\frac{ 1}{\left(\sum_{i=j}^{n}{N_{i}}+n-j+1-\sum_{i=j+1}^{n}{k_{i}}\right)}}}.
\end{array}
\end{equation}
It follows from Proposition~\ref{pro4} that
\begin{equation}
\begin{array}{ccc}
&\zeta_n(\underline{\alpha}; -\underline{\textbf{N}})=(-1)^n&\\
& \sum_{\underline{k}=(k_2,...,k_n) \in \mathbb{N}^{n-1}}{\sum_{\underline{v}=(v_1,...,v_n)\in \mathbb{N}^{n}\atop v_j\leq k_j\; \forall\; 2\leq j\leq n; v_1\leq \left(\sum_{i=1}^{n}{N_{i}}+n-\sum_{i=2}^{n}{k_{i}}\right)}A(-\underline{\textbf{N}})\; B_{\underline{v}} \prod_{j=1}^{n}{\frac{1 }{\left(\sum_{i=j}^{n}{N_{i}}+n-j+1-\sum_{i=j+1}^{n}{k_{i}}\right)}}}
\end{array}
\end{equation}
with
\begin{equation*}
B_{\underline{v}}= \prod_{j=1}^{n}{B_{v_{j}}}
\end{equation*}
and $B_{v_{j}}$ is the ${v_{j}}^{-th}$ Bernoulli number, which ends the proof of Theorem~\ref{theo2}.

%%%%%%%%%%%%%%%%%%%%%%%%%%%%%%%%%%%%%%%%%%%%%%%%%%%%%%%%%%%%%%%%%%%%%%%%%%%%%%

\section{Two special cases}
In this part, we give two applications of this results:
\subsection{Multiple Hurwitz zeta values at non positive integers}
In this paragraph, we give the values of the classical multiple zeta values at non positive integers.

So, for $\underline{\alpha}:=(\alpha,\dots,\alpha)$, where $\alpha \neq 0,-1,-2,\dots,$, we find
\begin{equation}
\zeta_n(\underline{\alpha}; \underline{\textbf{s}})= \zeta_n( \alpha; s_{1}, \dots, s_{n})= \sum_{\underline{m}=(m_1,\dots,m_n) \in \mathbb{N}^{n}}{\prod_{i=1}^{n}{\frac{1}{(m_{1}+\dots +m_{i}+\alpha)^{s_{i}}}}}
\end{equation}
which is the multiple Hurwitz zeta function.

Thus, if we apply Theorem~\ref{theo2}, we find the same result given in~\cite{sadaoui19}.
\begin{corollary}\
	
	%\label{theo2}
	Let $\underline{\textbf{N}}= (N_{1},\dots,N_{n})$ a point of $\mathbb{N}^{n}$, if  the point $(\underline{\textbf{s}}=-\underline{\textbf{N}})$ is not a polar divisor for the integral function $Y_n(\alpha; \underline{\textbf{s}})$, then the value of the multiple Hurwitz zeta function $\zeta_n(\alpha; \underline{\textbf{s}})$ at the point $(\underline{\textbf{s}}=-\underline{\textbf{N}})$ exists and is given by
	\begin{equation}
	\begin{array}{ccc}
	&\zeta_n(\alpha; -\underline{\textbf{N}})=(-1)^n&\\
	& \sum_{\underline{k}=(k_2,...,k_n) \in \mathbb{N}^{n-1}}{\sum_{\underline{v}=(v_1,...,v_n)\in \mathbb{N}^{n}\atop v_j\leq k_j\; \forall\; 2\leq j\leq n; v_1\leq \left(\sum_{i=1}^{n}{N_{i}}+n-\sum_{i=2}^{n}{k_{i}}\right)}A(-\underline{\textbf{N}})\; B_{\underline{v}} \prod_{j=1}^{n}{\frac{1 }{\left(\sum_{i=j}^{n}{N_{i}}+n-j+1-\sum_{i=j+1}^{n}{k_{i}}\right)}}}
	\end{array}
	\end{equation}
	with
	\begin{equation}
	A(-\underline{\textbf{N}})= \left(\sum_{i=1}^{n}{N_{i}}+n-\sum_{i=2}^{n}{k_{i}}\atop v_1\right) \alpha^{\left(\sum_{i=1}^{n}{N_{i}}+n-\sum_{i=2}^{n}{v_{i}}\right)}\prod_{j=2}^{n}{\left(\sum_{i=j}^{n}{N_{i}}+n-j+1-\sum_{i=j+1}^{n}{k_{i}} \atop k_{j} \right)\left(k_j \atop v_j\right) }.
	\end{equation}
	and
	\begin{equation*}
	T(\underline{\textbf{N}}):= \left\{\underline{k}=(k_2,...,k_n) \in \mathbb{N}^{n-1}: \quad 0\leq k_{j}\leq \sum_{i=j}^{n}{N_{i}}+n-j+1-\sum_{i=j+1}^{n}{k_{i}}, \; \forall\; 2\leq j\leq n\right\}.
	\end{equation*}
	and
	\begin{equation*}
	B_{\underline{v}}= \prod_{j=1}^{n}{B_{v_{j}}}
	\end{equation*}
	where $B_{v_{j}}$ is the ${v_{j}}^{-th}$ Bernoulli number.
\end{corollary}

\subsection{Multiple zeta values at non positive integers}
Now, for $\underline{\alpha}:=(1,\dots,1)$,  we find
\begin{eqnarray}
\zeta_n(\underline{\alpha}; \underline{\textbf{s}})&=& \zeta_n( 1; s_{1}, \dots, s_{n})=\zeta_n(\underline{\textbf{s}}) =\sum_{\underline{m}=(m_1,\dots,m_n) \in \mathbb{N}^{n}}{\prod_{i=1}^{n}{\frac{1}{(m_{1}+\dots +m_{i}+1)^{s_{i}}}}}\\
&=&\sum_{\underline{m}=(m_1,\dots,m_n) \in \mathbb{N}^{*n}}{\prod_{i=1}^{n}{\frac{1}{(m_{1}+\dots +m_{i})^{s_{i}}}}}
\end{eqnarray}
which is the multiple zeta function.

Thus, if we apply Theorem~\ref{theo2}, we find the same result given in~\cite{sadaoui14}.
\begin{corollary}\
	
Let $\underline{\textbf{N}}= (N_{1},\dots,N_{n})$ a point of $\mathbb{N}^{n}$, if  the point $(\underline{\textbf{s}}=-\underline{\textbf{N}})$ is not a polar divisor for the integral function $Y(\underline{\textbf{s}})$, then the value of the multiple zeta function $Z(\underline{\textbf{s}})$ at the point $(\underline{\textbf{s}}=-\underline{\textbf{N}})$ exists and is given by~\footnote{This corrects a typo in\cite[Theorem 1]{sadaoui14}}.
\begin{equation}
\begin{array}{ccc}
&\zeta_n(-\underline{\textbf{N}})=(-1)^{n}&\\
& \sum_{\underline{k}=(k_2,...,k_n) \in T(\underline{\textbf{N}})}{\sum_{\underline{v}=(v_1,...,v_n)\in \mathbb{N}^{n}\atop v_j\leq k_j\; \forall\; 2\leq j\leq n; v_1\leq \left(\sum_{i=1}^{n}{N_{i}}+n-\sum_{i=2}^{n}{k_{i}}\right)}A(-\underline{\textbf{N}})\; B_{\underline{v}}\prod_{j=1}^{n}{\frac{1 }{\left(\sum_{i=j}^{n}{N_{i}}+n-j+1-\sum_{i=j+1}^{n}{k_{i}}\right)}}}&
\end{array}
\end{equation}
with
\begin{equation*}
A(-\underline{\textbf{N}})= \left(\sum_{i=1}^{n}{N_{i}}+n-\sum_{i=2}^{n}{k_{i}}\atop v_1\right)\prod_{j=2}^{n}{\left(\sum_{i=j}^{n}{N_{i}}+n-j+1-\sum_{i=j+1}^{n}{k_{i}} \atop k_{j} \right)\left(k_j \atop v_j\right)}
\end{equation*}

\begin{equation*}
T(\underline{\textbf{N}}):= \left\{\underline{k}=(k_2,...,k_n) \in \mathbb{N}^{n-1}: \quad 0\leq k_{j}\leq \sum_{i=j}^{n}{N_{i}}+n-j+1-\sum_{i=j+1}^{n}{k_{i}}, \; \forall\; 2\leq j\leq n\right\}.
\end{equation*}
and
\begin{equation*}
B_{\underline{v}}= \prod_{j=1}^{n}{B_{v_{j}}}
\end{equation*}
where $B_{v_{j}}$ is the ${v_{j}}^{-th}$ Bernoulli number.	
\end{corollary}

%%%%%%%%%%%%%%%%%%%%%%%%%%%%%%%%%%%%%%%%%%%%%%%%%%%%%%%%%%%%%%%%%%%%%%%%

\section*{Acknowledgement}

We thank the referees for their numerous and very helpful comments and suggestions which greatly contributed in improving the final presentation.

%%%%%%%%%%%%%%%%%%%%%%%%%%%%%%%%%%%%%%%%%%%%%%%%%%%%%%%%%%%%%%%%%%%%%%%%
%\newpage

\end{document}